\documentclass[11pt,a4paper]{article}
\usepackage{epsf,epsfig,amsfonts,amsgen,amsmath,amstext,amsbsy,amsopn,amsthm}
\usepackage{amsmath}
\usepackage{enumerate}
\usepackage{hhline}
\usepackage{multirow}
\usepackage{multicol}
\usepackage{booktabs}
\usepackage{lineno}
\usepackage{amsfonts,amsthm,amssymb,bm}
\usepackage{amsfonts}
\usepackage{graphics}
\usepackage{latexsym,bm}
\usepackage{amsfonts,amsthm,amssymb,bbding}
\usepackage{indentfirst}
\usepackage{graphicx}
\usepackage{color}
\usepackage[colorlinks=true,anchorcolor=blue,filecolor=blue,linkcolor=red,urlcolor=blue,citecolor=blue]{hyperref}
\usepackage{float}
\usepackage{tikz,enumerate}
\usetikzlibrary{calc}
\usepackage{geometry}
\newcommand{\n}{\noindent}

\newtheorem{thm}{Theorem}[section]

\newtheorem{lemma}[thm]{Lemma}

\newtheorem{pro}[thm]{Problem}

\newtheorem{cor}[thm]{Corollary}

\newtheorem{prop}[thm]{Proposition}

\newtheorem{obs}{Observation}[section]

\renewcommand{\le}{\leqslant}
\renewcommand{\geq}{\geqslant}
\renewcommand{\ge}{\geqslant}

\newcommand{\calG}{\mathcal{G}}

\newcommand{\N}{\mathbb{N}}

\begin{document}

\title{A non-hereditary Pollyanna class that is not strongly Pollyanna}

\author{
Hongzhang Chen\thanks{School of Mathematics and Statistics, Gansu
Center for Applied Mathematics, Lanzhou University, Lanzhou, Gansu,
730000, China. Email: \url{mnhzchern@gmail.com}.} 
\and 
Kaiyang Lan\thanks{Corresponding author. School of Mathematics and Statistics, Minnan Normal University, Zhangzhou, Fujian, 363000, China. Email: \url{kylan95@126.com}. Partially supported by the Youth Foundation of Fujian Province (Grant No. JZ240035) and the Minnan Normal University Foundation (Grant No. KJ2023002).} 
 }

\date{\today}
\maketitle



\begin{abstract}
	Chudnovsky, Cook, Davies, and Oum introduced the notion of \textit{Pollyanna} graph classes: a class $\mathcal{C}$ is Pollyanna if for every $\chi$-bounded class $\mathcal{F}$, the intersection $\mathcal{C} \cap \mathcal{F}$ is polynomially $\chi$-bounded.
	They further defined $\mathcal{C}$ to be \textit{strongly Pollyanna} if it is $k$-strongly Pollyanna for some integer $k$, meaning that $\mathcal{C} \cap \mathcal{F}$ is polynomially $\chi$-bounded for every $k$-good class $\mathcal{F}$.
	They asked whether there are Pollyanna graph classes that are not strongly Pollyanna.
	In this note we answer this question affirmatively, under the literal interpretation that graph classes are not required to be hereditary.
	We construct a class $\mathcal{C}$ that is Pollyanna but, for every $k \ge 1$, is not $k$-strongly Pollyanna; in particular $\mathcal{C}$ is not strongly Pollyanna.
\end{abstract}

{\bf Keywords: Pollyanna class, $\chi$-boundedness, polynomial $\chi$-boundedness} 

{\bf 2020 AMS Subject Classifications: 05C15, 05C75} 

\section{Introduction}

All graphs in this note are finite, simple, and undirected.
The \emph{chromatic number} $\chi(G)$ of a graph $G$ is the minimum number of colors needed to color its vertices so that no two adjacent vertices share the same color. A \emph{clique} in $G$ is a set of pairwise adjacent vertices, and the \emph{clique number} $\omega(G)$ is the size of a largest clique in $G$. It is clear that $\chi(G) \geq \omega(G)$, but in general $\chi(G)$ cannot be bounded from above by any function of $\omega(G)$; for instance, there are triangle-free graphs with arbitrarily large chromatic number~\cite{Erdos1959,Mycielski1955,Zykov1949}.
Empty induced subgraphs will not play any role; when they occur formally, we use the harmless conventions $\chi(\emptyset)=\omega(\emptyset)=0$.

A graph class $\mathcal{F}$ is \emph{$\chi$-bounded}~\cite{Gyarfas1975} if there exists a function $\varphi: \mathbb{N} \to \mathbb{N}$ such that for every $G \in \mathcal{F}$ and every induced subgraph $H$ of $G$, we have $\chi(H) \le \varphi(\omega(H))$.
Such a function $\varphi$ is called a \emph{$\chi$-bounding function} for $\mathcal{F}$.
The class $\mathcal{F}$ is \emph{polynomially $\chi$-bounded} if $\varphi$ can be chosen to be bounded above by a polynomial function in $\omega(H)$; equivalently, there exists a polynomial function $p$ such that $\chi(H) \le p(\omega(H))$ for every $G \in \mathcal{F}$ and every induced subgraph $H$ of $G$.
Esperet~\cite{Esperet2017} conjectured that every $\chi$-bounded class is polynomially $\chi$-bounded. This conjecture was recently disproved by Bria\'nski, Davies, and Walczak~\cite{BrianskiDaviesWalczak2024}, who constructed $\chi$-bounded classes that are not polynomially $\chi$-bounded.

Nevertheless, inspired by Esperet's conjecture, Chudnovsky, Cook, Davies, and Oum~\cite{CCDO2026} introduced the notion of \emph{Pollyanna} classes. 
A class $\mathcal{C}$ of graphs is called \emph{Pollyanna} if $\mathcal{C} \cap \mathcal{F}$ is polynomially $\chi$-bounded for every $\chi$-bounded class $\mathcal{F}$. 
In other words, $\mathcal{C}$ is Pollyanna if Esperet's conjecture holds within $\mathcal{C}$. 
Clearly, every polynomially $\chi$-bounded class is Pollyanna. 

A key notion in \cite{CCDO2026} is that of $n$-good classes. 
For an integer $n$, let $\chi^{(n)}(G)$ be the maximum chromatic number of an induced subgraph $H$ of $G$ with $\omega(H) \le n$. 
A class $\mathcal{F}$ of graphs is called $n$-\emph{good} if it is hereditary and there exists a constant $m$ (depending on $\mathcal{F}$ and $n$) such that $\chi^{(n)}(G) \le m$ for every $G \in \mathcal{F}$.
Note that $n$-goodness is a strictly weaker condition than $\chi$-boundedness.
Using this, a stronger property is defined in~\cite{CCDO2026}: a class $\mathcal{C}$ is \emph{$n$-strongly Pollyanna} if $\mathcal{C} \cap \mathcal{F}$ is polynomially $\chi$-bounded for every $n$-good class $\mathcal{F}$.
A class $\mathcal{C}$ is \emph{strongly Pollyanna} if it is $n$-strongly Pollyanna for some $n$.

Theorem 1.2 of~\cite{CCDO2026} shows that several natural graph classes are (strongly) Pollyanna, while Theorem 1.3 shows that several proper classes of graphs are not Pollyanna.
Section 10 of~\cite{CCDO2026} proposes several problems that remain open.
One of them is the following.

\begin{pro}[\cite{CCDO2026}]\label{problem10.1}
	Are there Pollyanna graph classes that are not strongly Pollyanna?
\end{pro}

The main result of this note answers Problem~\ref{problem10.1} affirmatively under the literal convention that graph classes are not required to be hereditary.

\begin{thm}\label{thm:main}
	There exists a graph class $\mathcal{C}$ such that $\mathcal{C}$ is Pollyanna but $\mathcal{C}$ is not $k$-strongly Pollyanna for any integer $k \ge 1$. In particular, $\mathcal{C}$ is not strongly Pollyanna. The class $\mathcal{C}$ constructed here is not hereditary.
\end{thm}

The construction has two ingredients. 
First, for each integer $r \ge 2$, we use a theorem of Bri{\'n}ski, Davies, and Walczak~\cite{BrianskiDaviesWalczak2024} to obtain a family of graphs with clique number exactly $r$ and arbitrarily large chromatic number, while every induced subgraph with clique number at most $r-1$ has bounded chromatic number. 
Second, we take a triangle-free graph $T_r$ of chromatic number $r$ (a Mycielski graph~\cite{Mycielski1955}) and attach it as a disjoint copy to each graph in the family. 
The tag $T_r$ does not affect the clique number (since it is triangle-free) but reveals the parameter $r$ to any $\chi$-bounding function, because $\chi(T_r)=r$ forces $r$ to be bounded when the class is intersected with a $\chi$-bounded test class.

\section{Tools}\label{sec:tools}

This section records the two external or classical inputs used in the construction.
The first is the theorem of Bria\'nski, Davies, and Walczak \cite{BrianskiDaviesWalczak2024}.

For a graph class $\calG$ and an integer $n\ge1$, define
\[
\chi_{\max}(\calG,n)
=
\sup\{\chi(G):G\in\calG \text{ and } \omega(G)=n\}
\in \N\cup\{\infty\}.
\]
If $\chi_{\max}(\mathcal{G}, n) = \infty$, this means that graphs in $\mathcal{G}$ with clique number exactly $n$ have unbounded chromatic number.


\begin{thm}[Briański--Davies--Walczak \cite{BrianskiDaviesWalczak2024}]\label{thm:BDW}
	Let $f : \mathbb{N} \to \mathbb{N} \cup \{\infty\}$ be such that $f(1) = 1$ and $f(n) \geqslant \binom{3n+1}{3}$ for every $n \geqslant 2$. Then there exists a hereditary class of graphs $\mathcal{G}$ such that $\chi_{\max}(\calG,n) = f(n)$ for every $n \in \mathbb{N}$.
\end{thm}

We shall use Theorem~\ref{thm:BDW} only through the following consequence.

\begin{cor}\label{cor:Grm}
	For every integer $r \ge 2$ there exists a finite constant $B_r$ and graphs
	\[
	G_{r,1}, G_{r,2}, G_{r,3}, \ldots
	\]
	such that, for every $m \ge 1$,
	\begin{enumerate}[(i)]
		\item $\omega(G_{r,m}) = r$;
		\item $\chi(G_{r,m}) \ge m$;
		\item if $H$ is an induced subgraph of $G_{r,m}$ and $\omega(H) \le r-1$, then $\chi(H) \le B_r$.
	\end{enumerate}
\end{cor}
\begin{proof}
	Fix $r \ge 2$. Define
	\[
	f_r(1) = 1,
	\qquad
	f_r(n) = \binom{3n+1}{3} \;\; \text{for } 2 \le n < r,
	\qquad
	\text{and~} f_r(n) = \infty \;\; \text{for } n \ge r.
	\]
	The function $f_r$ satisfies the hypotheses of Theorem~\ref{thm:BDW}. Hence there exists a hereditary class $\mathcal{G}_r$ such that \(\chi_{\max}(\mathcal{G}_r, n) = f_r(n)\) for every $n \ge 1$.
	Since $f_r(r) = \infty$, the chromatic numbers of graphs in $\mathcal{G}_r$ with clique number exactly $r$ are unbounded. Therefore, for every $m \ge 1$, we may choose a graph $G_{r,m} \in \mathcal{G}_r$ such that \(\omega(G_{r,m}) = r\) and \(\chi(G_{r,m}) \ge m\).
This proves (i) and (ii).
	
	It remains to prove (iii).
	Set \(B_r = \max\{ f_r(i) : 1 \le i \le r-1 \}\).
	This is finite because $f_r(i) < \infty$ for $1 \le i \le r-1$.
	Let $H$ be an induced subgraph of $G_{r,m}$ and suppose $\omega(H) \le r-1$. Since $\mathcal{G}_r$ is hereditary and $G_{r,m} \in \mathcal{G}_r$, we have $H \in \mathcal{G}_r$. Let $i = \omega(H)$. Then $1 \le i \le r-1$ unless $H$ is empty, and the empty case is trivial. By the definition of $\chi_{\max}$, \(\chi(H) \le \chi_{\max}(\mathcal{G}_r, i) = f_r(i) \le B_r\).
	This proves (iii).
\end{proof}

The second ingredient is the elementary Mycielski construction~\cite{Mycielski1955}, which we include with proof to make the tag graphs explicit.
Specifically, for every $r \ge 2$ we need a triangle-free graph of chromatic number exactly $r$. 
The following standard construction supplies such graphs explicitly.

Let $G$ be a graph with vertex set $V(G) = \{v_1, \ldots, v_n\}$. 
The \textit{Mycielskian} $M(G)$ is the graph with vertex set \(\{v_1, \ldots, v_n\} \cup \{u_1, \ldots, u_n\} \cup \{w\}\) and with the following edges:
\begin{itemize}
	\item all edges $v_i v_j$ that occur in $G$;
	\item for every edge $v_i v_j$ of $G$, the two edges $u_i v_j$ and $u_j v_i$;
	\item the edges $w u_i$ for all $i \in \{1, \ldots, n\}$.
\end{itemize}
There are no other edges.
In particular, the vertices $u_1, \ldots, u_n$ are pairwise non-adjacent, and $w$ is adjacent only to the vertices $u_i$.

\begin{lemma}[Mycielski \cite{Mycielski1955}]\label{lem:mycielski}
	If $G$ is triangle-free and $\chi(G) = q \ge 2$, then $M(G)$ is triangle-free and $\chi(M(G)) = q + 1$.
	Consequently, for every integer $r \ge 2$ there exists a connected triangle-free graph $T_r$ such that $\omega(T_r) = 2$ and $\chi(T_r) = r$.
\end{lemma}
\begin{proof}
	First we prove that $M(G)$ is triangle-free when $G$ is triangle-free. 
	A triangle in $M(G)$ that contains $w$ would need two neighbors of $w$ that are adjacent to each other. 
	But the neighbors of $w$ are precisely $u_1,\ldots,u_n$, and these vertices are pairwise non-adjacent. 
	Thus no triangle contains $w$.
	
	Now consider a triangle not containing $w$. 
	If it uses only vertices of the form $v_i$, then it is already a triangle of $G$, impossible. 
	If it uses two vertices of the form $u_i$, then it cannot be a triangle because no two vertices in \(\{u_1,u_2,\ldots,u_n\}\) are adjacent. 
	Hence the only remaining possibility is a triangle of the form $u_i v_j v_\ell$.
	For $u_i$ to be adjacent to both $v_j$ and $v_\ell$, the edges $v_i v_j$ and $v_i v_\ell$ must be in $G$ (by the definition of $M(G)$). 
	Since $v_j v_\ell$ is also an edge of the alleged triangle, the vertices $v_i, v_j, v_\ell$ form a triangle in $G$, again impossible. 
	Therefore $M(G)$ is triangle-free.
	
	Next we prove the chromatic number formula. 
	Let $c$ be a proper $q$-coloring of $G$. 
	Color each old vertex $v_i$ of $M(G)$ by $c(v_i)$, color each new vertex $u_i$ by $c(v_i)$, and color $w$ with one additional color. 
	This is a proper coloring of $M(G)$ using $q+1$ colors. 
	Indeed, if $u_i$ is adjacent to $v_j$, then $v_i$ is adjacent to $v_j$ in $G$, so $c(v_i) \ne c(v_j)$. 
	Thus $\chi(M(G)) \le q+1$.
	
	For the reverse inequality, suppose that $M(G)$ has a proper coloring with $s$ colors. 
	Relabel the colors so that $w$ has color $s$. 
	Since $w$ is adjacent to every $u_i$, no $u_i$ has color $s$. 
	We construct a coloring $c'$ of $G$ using only the colors $1,\ldots,s-1$ as follows:
	\[
	c'(v_i) = 
	\begin{cases}
		c(u_i), & \text{if } c(v_i) = s,\\
		c(v_i), & \text{if } c(v_i) \ne s.
	\end{cases}
	\]
	We check that $c'$ is proper. 
	Let $v_i v_j \in E(G)$. 
	If neither $v_i$ nor $v_j$ has color $s$ under $c$, then $c'(v_i)=c(v_i)$ and $c'(v_j)=c(v_j)$, which are distinct because $v_i$ and $v_j$ are adjacent in the old copy of $G$ inside $M(G)$. 
	If, say, $c(v_i)=s$ and $c(v_j)\ne s$, then $c'(v_i)=c(u_i)$ and $c'(v_j)=c(v_j)$. 
	The vertices $u_i$ and $v_j$ are adjacent in $M(G)$ because $v_i$ and $v_j$ are adjacent in $G$. 
	Hence $c(u_i) \ne c(v_j)$, so $c'(v_i) \ne c'(v_j)$. 
	The symmetric case is the same. 
	Finally, both $c(v_i)=s$ and $c(v_j)=s$ cannot occur because $v_i$ and $v_j$ are adjacent in $M(G)$. 
	Therefore $c'$ is a proper $(s-1)$-coloring of $G$. 
	Since $\chi(G)=q$, we get $s-1 \ge q$, so $s \ge q+1$. 
	Hence $\chi(M(G)) \ge q+1$.
	Together with the upper bound, this gives $\chi(M(G)) = q+1$.
	
	To obtain $T_r$, set $T_2 = K_2$ and define recursively $T_{r+1} = M(T_r)$ for $r \ge 2$. 
	The graph $K_2$ is triangle-free and has chromatic number $2$. 
	By the two properties just proved, induction gives that $T_r$ is triangle-free and $\chi(T_r) = r$ for every $r \ge 2$. 
	Since $T_r$ contains an edge and is triangle-free, its clique number is $2$. 
	The same induction also gives connectedness. 
	$K_2$ is connected, and if $G$ is connected with at least one edge, then every vertex $u_i$ of $M(G)$ is adjacent to $w$, while every old vertex $v_i$ has a neighbor $v_j$ in $G$ and hence is adjacent to $u_j$ in $M(G)$.
	Thus $M(G)$ is connected.
	This proves Lemma~\ref{lem:mycielski}.
\end{proof}

\section{Proofs}

If $G$ and $H$ are graphs, then $G \cup H$ denotes their disjoint union.

For every integer $r \ge 2$, fix graphs $G_{r,1}, G_{r,2}, \ldots$ and a constant $B_r$ as in Corollary~\ref{cor:Grm}.
Also fix a triangle-free graph $T_r$ with $\omega(T_r) = 2$ and $\chi(T_r) = r$, as supplied by Lemma~\ref{lem:mycielski}.
For every $r \ge 2$ and every $m \ge 1$, define $X_{r,m} = G_{r,m} \cup T_r$.
Let $\mathcal{C}_r = \{ X_{r,m} : m \ge 1 \}$ and finally set $\mathcal{C} = \bigcup_{r \ge 2} \mathcal{C}_r$.

The following immediate facts are the heart of the construction.

\begin{obs}\label{obs:block}
	For every $r \ge 2$ and every $m \ge 1$,
	\begin{enumerate}[(i)]
		\item $\omega(X_{r,m}) = r$;
		\item $\chi(X_{r,m}) \ge m$;
		\item $X_{r,m}$ contains $T_r$ as an induced subgraph, and this induced subgraph satisfies $\omega(T_r) = 2$ and $\chi(T_r) = r$.
	\end{enumerate}
\end{obs}
\begin{proof}
	The clique number of a disjoint union is the maximum of the clique numbers of its components. Therefore, by Corollary~\ref{cor:Grm}
	\[
	\omega(X_{r,m}) = \max\{\omega(G_{r,m}), \omega(T_r)\} = \max\{r, 2\} = r,
	\]
	since $r \ge 2$. The chromatic number of a disjoint union is likewise the maximum of the chromatic numbers of its components, so, by Corollary~\ref{cor:Grm}
	\[
	\chi(X_{r,m}) = \max\{\chi(G_{r,m}), \chi(T_r)\} \ge \chi(G_{r,m}) \ge m.
	\]
	Finally, $T_r$ is a component of $X_{r,m}$, hence an induced subgraph, and Lemma~\ref{lem:mycielski} gives $\omega(T_r) = 2$ and $\chi(T_r) = r$.
	This proves Observation~\ref{obs:block}.
\end{proof}

\noindent\textit{Remark.} The class $\mathcal{C}$ is plainly not hereditary: for example, $T_r$ is an induced subgraph of every $X_{r,m}$, but $T_r \notin \mathcal{C}$, since every member of $\mathcal{C}$ is disconnected (as a disjoint union of two non-empty graphs), whereas the Mycielski graphs $T_r$ are connected.
The study of $\chi$-bounded graph classes asks when the chromatic number can be controlled by the clique number throughout a class. We use the non-hereditary convention used by Chudnovsky, Cook, Davies, and Oum \cite{CCDO2026}. Thus a class need not be hereditary, but the bound is required to hold for all induced subgraphs of its members.

\medskip

We now prove the first half of Theorem~\ref{thm:main}.

\begin{prop}\label{prop:pollyanna}
	The class $\mathcal{C}$ is Pollyanna.
\end{prop}
\begin{proof}
	Let $\mathcal{F}$ be an arbitrary $\chi$-bounded graph class. Choose a $\chi$-bounding function \(\varphi : \mathbb{N} \to \mathbb{N}\) for $\mathcal{F}$.
	Thus, whenever $Y \in \mathcal{F}$ and $Z$ is an induced subgraph of $Y$, we have \(\chi(Z) \le \varphi(\omega(Z))\).
	(Here no monotonicity of $\varphi$ is assumed or needed.)
	
	We must show that $\mathcal{C} \cap \mathcal{F}$ is polynomially $\chi$-bounded. If $\mathcal{C} \cap \mathcal{F} = \emptyset$, there is nothing to prove. Otherwise, take any graph \(X_{r,m} \in \mathcal{C} \cap \mathcal{F}\).
	By Observation~\ref{obs:block}, $T_r$ is an induced subgraph of $X_{r,m}$ with \(\omega(T_r) = 2\) and \(\chi(T_r) = r\).
	Since $X_{r,m} \in \mathcal{F}$ and $\varphi$ bounds the chromatic number of every induced subgraph of every graph in $\mathcal{F}$, we obtain \(r = \chi(T_r) \le \varphi(\omega(T_r)) = \varphi(2)\).
	Thus every member of $\mathcal{C} \cap \mathcal{F}$ belongs to one of the finitely many blocks \(\mathcal{C}_2, \mathcal{C}_3, \ldots, \mathcal{C}_{\varphi(2)}\).
	In particular, the parameter $r$ is bounded solely in terms of the test class $\mathcal{F}$.
	
	Next, we use the fact that $X_{r,m}$ itself belongs to $\mathcal{F}$. By Observation~\ref{obs:block}, $\omega(X_{r,m}) = r$. Hence the $\chi$-bounding inequality for $\mathcal{F}$ gives \(\chi(X_{r,m}) \le \varphi(\omega(X_{r,m})) = \varphi(r)\).
	Because $r \le \varphi(2)$, this is bounded by the finite constant \(M_{\mathcal{F}} = \max\{\varphi(s) : 1 \le s \le R_{\mathcal{F}}\}\), where \(R_{\mathcal{F}} = \max\{2, \varphi(2)\}\).
	Thus every graph $X_{r,m} \in \mathcal{C} \cap \mathcal{F}$ satisfies \(\chi(X_{r,m}) \le M_{\mathcal{F}}\).

	Now let $Z$ be an induced subgraph of some $X_{r,m} \in \mathcal{C} \cap \mathcal{F}$. Since the chromatic number does not increase when passing to an induced subgraph, we have \(\chi(Z) \le \chi(X_{r,m}) \le M_{\mathcal{F}}\).
	Consequently, the constant polynomial \(p_{\mathcal{F}}(x) = M_{\mathcal{F}}\) is a polynomial $\chi$-bounding function for $\mathcal{C} \cap \mathcal{F}$.
	Since $\mathcal{F}$ was arbitrary, $\mathcal{C}$ is Pollyanna.	
	This proves Proposition~\ref{prop:pollyanna}.
\end{proof}

\noindent\textit{Remark.}
The proof shows explicitly why the tags $T_r$ are useful: 
a $\chi$-bounded class cannot contain tags of clique number $2$ 
and arbitrarily large chromatic number. 
Therefore, when intersected with a $\chi$-bounded class, 
the union over all $r$ collapses to only finitely many $r$-levels.

It remains to prove that no fixed $k$ suffices for strong Pollyanna-ness. 
The witness against $k$-strong Pollyanna-ness will be the hereditary closure of a single block $\mathcal{C}_r$ with $r = k+1$.

For a graph class $\mathcal{A}$, write
\[
\operatorname{Ind}(\mathcal{A}) = \{ H : H \text{ is an induced subgraph of some } G \in \mathcal{A} \}
\]
for its hereditary closure.

We begin with:

\begin{lemma}\label{lem:blockgood}
	For every integer $r \ge 2$, the hereditary class \(\mathcal{H}_r = \operatorname{Ind}(\mathcal{C}_r)\) is $(r-1)$-good.
\end{lemma}
\begin{proof}
	The class $\mathcal{H}_r$ is hereditary by definition. 
	We must prove that there is a constant $W_r$ such that \(\chi^{(r-1)}(Y) \le W_r\) for every $Y \in \mathcal{H}_r$.
	
	Let \(W_r = \max\{ B_r, r \}\), where $B_r$ is the constant from Corollary~\ref{cor:Grm}.
	We claim that this $W_r$ works.
	Take any $Y \in \mathcal{H}_r$.
	By definition of $\mathcal{H}_r$, there exists $m \ge 1$ such that \(Y\) is an induced subgraph of \(X_{r,m} = G_{r,m} \cup T_r\).
	To bound $\chi^{(r-1)}(Y)$, let $Z$ be any induced subgraph of $Y$ with $\omega(Z) \le r-1$. 
	Then $Z$ is also an induced subgraph of $X_{r,m}$. 
	Since $X_{r,m}$ is a disjoint union of $G_{r,m}$ and $T_r$, the graph $Z$ decomposes as a disjoint union \(Z = Z_G \cup Z_T\), where $Z_G$ is an induced subgraph of $G_{r,m}$ and $Z_T$ is an induced subgraph of $T_r$.
	
	The clique number of each component of a disjoint union is at most the clique number of the whole union. Hence \(\omega(Z_G) \le \omega(Z) \le r-1\).
	By Corollary~\ref{cor:Grm} (iii), this implies $\chi(Z_G) \le B_r$. 
	Also, $\chi(Z_T) \le \chi(T_r) = r$.
	Analogously, the chromatic number of a disjoint union is the maximum of the chromatic numbers of its components, so \(\chi(Z) = \max\{\chi(Z_G), \chi(Z_T)\} \le \max\{B_r, r\} = W_r\).
	This bound holds for every induced subgraph $Z$ of $Y$ with $\omega(Z) \le r-1$.
	Therefore \(\chi^{(r-1)}(Y) \le W_r\).
	Since $Y \in \mathcal{H}_r$ was arbitrary, $\mathcal{H}_r$ is $(r-1)$-good.
	This proves Lemma~\ref{lem:blockgood}.
\end{proof}

\begin{lemma}\label{lem:blockbad}
	For every integer $r \ge 2$, the class \(\mathcal{C} \cap \mathcal{H}_r\) is not polynomially $\chi$-bounded.
\end{lemma}
\begin{proof}
	Since every graph in $\mathcal{C}_r$ belongs to its hereditary closure $\mathcal{H}_r$, we have \(\mathcal{C}_r \subseteq \mathcal{C} \cap \mathcal{H}_r\).
	For every $m \ge 1$, Observation~\ref{obs:block} gives \(\omega(X_{r,m}) = r\) and \(\chi(X_{r,m}) \ge m\).
	Thus $\mathcal{C} \cap \mathcal{H}_r$ contains graphs of fixed clique number $r$ and arbitrarily large chromatic number.
	
	Suppose, for contradiction, that $\mathcal{C} \cap \mathcal{H}_r$ were polynomially $\chi$-bounded. 
	Then there would exist a polynomial $p$ such that every induced subgraph \(H\) of every graph \(G\) in $\mathcal{C} \cap \mathcal{H}_r$ has chromatic number at most $p(\omega(H))$. 
	Applying this to the graphs $X_{r,m}$ themselves (each is an induced subgraph of itself) yields \(\chi(X_{r,m}) \le p(\omega(X_{r,m})) = p(r)\) for every $m \ge 1$. 
	But $p(r)$ is a fixed finite number, while $\chi(X_{r,m}) \ge m$ is unbounded as $m \to \infty$. 
	Choosing $m > p(r)$ gives a contradiction. 
	Hence $\mathcal{C} \cap \mathcal{H}_r$ is not polynomially $\chi$-bounded.
	This proves Lemma~\ref{lem:blockbad}.
\end{proof}

We can now finish the proof of the main theorem.

\begin{proof}[\bf Proof of Theorem \ref{thm:main}]
	For every integer $r \ge 2$, fix graphs \(G_{r,1}, G_{r,2}, \ldots\) and a constant $B_r$ as in Corollary~\ref{cor:Grm}.
	Also fix a triangle-free graph $T_r$ with \(\omega(T_r) = 2\) and \(\chi(T_r) = r\), as supplied by Lemma~\ref{lem:mycielski}.
	For every $r \ge 2$ and every $m \ge 1$, define \(X_{r,m} = G_{r,m} \cup T_r\).
	Let \(\mathcal{C}_r = \{ X_{r,m} : m \ge 1 \}\) and finally set \(\mathcal{C} = \bigcup_{r \ge 2} \mathcal{C}_r\).
	Proposition~\ref{prop:pollyanna} shows that $\mathcal{C}$ is Pollyanna.
	It remains to prove that $\mathcal{C}$ is not strongly Pollyanna.
	
	Let $k \ge 1$ be arbitrary and set $r = k+1$ (so $r \ge 2$). 
	By Lemma~\ref{lem:blockgood}, the hereditary class \(\mathcal{H}_r = \operatorname{Ind}(\mathcal{C}_r)\)
	is $(r-1)$-good, i.e., $k$-good. 
	However, by Lemma~\ref{lem:blockbad}, the intersection \(\mathcal{C} \cap \mathcal{H}_r\) is not polynomially $\chi$-bounded. 
	Therefore $\mathcal{C}$ is not $k$-strongly Pollyanna.
	Since $k \ge 1$ was arbitrary, $\mathcal{C}$ is not $k$-strongly Pollyanna for any $k$; hence $\mathcal{C}$ is not strongly Pollyanna. 
	Finally, as observed after the construction, $\mathcal{C}$ is not hereditary.
	This proves Theorem \ref{thm:main}.
\end{proof}

\vspace{6mm}

\n{\bf Acknowledgements:} 
The research is partially supported by Fujian Key Laboratory of Granular Computing and Applications (MinnanNormal University), Institute of Meteorological Big Data-Digital Fujian and Fujian Key Laboratory of Data Science and Statistics.

\paragraph{Data availability.}
Data sharing is not applicable to this article as no datasets were generated or
analysed during the current study.

\paragraph{Conflict of interest.}
The authors have no relevant financial or non-financial interests to disclose.

\end{document}